\documentclass[reqno]{amsart}
\usepackage{color,graphicx,enumerate,wrapfig,amssymb}
\usepackage{pxfonts}
\usepackage{subfigure}
\usepackage{eucal}

\newcommand{\R}{{\mathbb R}}
\newcommand{\Rn}{{\mathbb R}^n}

\newcommand{\N}{{\mathbb N}}

\newcommand{\re}{\mathbb{R}}

\newcommand\norm[1]{\Arrowvert {#1}\Arrowvert}

\newtheorem{theorem}{Theorem}[section]

\newtheorem{cor}[theorem]{Corollary}

\newtheorem{definition}{Definition}[section]

\newtheorem{lemma}[theorem]{Lemma}

\newtheorem{prop}[theorem]{Proposition}
\newtheorem{remark}[theorem]{Remark}

\title [Double Obstacle Problems]{The Regularity Theory for the Double Obstacle Problem for Fully Nonlinear Operator}

\author{Ki-ahm Lee}
\address{Department of Mathematical Sciences,
Seoul National University, Seoul 08826, Korea \& Korea Institute for Advanced Study, Seoul 02455, Korea}
\email{kiahm@snu.ac.kr}

\author{Jinwan Park}
\address{Research Institute of Mathematics, 
Seoul National University, Seoul 08826, Korea}
\email{jinwann@snu.ac.kr}


\begin{document}
\maketitle

\begin{abstract}
In this paper, we prove the existence and uniqueness of $W^{2,p}$ ($n<p<\infty$) solutions of a double obstacle problem with $C^{1,1}$ obstacle functions. Moreover, we show the optimal regularity of the solution and the local $C^1$ regularity of the free boundary. In the study of the regularity of the free boundary, we deal with a general problem, the no-sign reduced double obstacle problem with an upper obstacle $\psi$,
$F(D^2 u,x) =f\chi_{\Omega(u) \cap\{ u< \psi\} } + F(D^2\psi,x) \chi_{\Omega(u)\cap \{u=\psi\}}, u\le \psi  \text { in } B_1$,
where 
$\Omega(u)=B_1 \setminus \left( \{u=0\} \cap \{ \nabla u =0\}\right)$. %under thickness assumptions for $u$ and $\psi$. %The study on the regularity of the free boundary for the general problem is an extension of the research in \cite{LPS}.
\end{abstract}

\maketitle

\section{Introduction}

Obstacle problems with a single obstacle appear in various fields of study such as porous media, elasto-plasticity, optimal control, and financial mathematics, see \cite{Fri, Caf98}. The regularity of the solutions and the free boundaries of the problems have been actively studied by \cite{Caf77, CKS, Iva, Lee98, ALS}.

The double obstacle problem, which is the obstacle problem with two obstacles, originates in the study of optimal investment problems with transaction costs, the game of tug-of-war, and semiconductor devices, (see \cite{MR} and the references therein). Recently, global homogeneous solutions to the double obstacle problem, with homogeneous obstacles was considered by \cite{Ale} and the regularity of the free boundaries of the double obstacle problem for Laplacian was obtained by \cite{LPS}. 

In this paper, we discuss the regularity of the solution and the free boundary for the double obstacle problem of the fully nonlinear operator. Precisely, we prove the existence and uniqueness of $W^{2,p}$ ($n<p<\infty$) solutions of \emph{double obstacle problem} for the fully nonlinear operator in a domain $D\subset \re^n$,
\begin{equation}\label{eq-main-global}\tag{$FB$}
\left\{ \begin{array}{l l}
F(D^2u,x)\ge0, & \quad \text{ in } \{u>\phi_1\} \cap D,\\
F(D^2u,x)\le0, & \quad \text{ in } \{u<\phi_2\} \cap D,\\
\phi_1(x)\le u(x)\le \phi_2(x) & \quad \text{ in } D,\\
u(x)= g(x) & \quad \text{ on } \partial D, 
\end{array}\right.
\end{equation}
with $\phi_1, \phi_2\in C^{1,1}(\overline{D})$, $\partial D \in C^{2,\alpha}$, $g\in C^{2,\alpha}(\overline{D})$ and $\phi_1\le g \le \phi_2$ in $\partial D$. The optimal ($C^{1,1}$) regularity of the solution $u$ to \eqref{eq-main-global} is also obtained. Moreover, we have $C^{1}$ regularity of the free boundary $\partial \{ u=\psi_1\}$ of \eqref{eq-main-global}, by studying the regularity of the free boundary for a general problem \eqref{eq-main-local}, which contains a reduced version \eqref{eq-ori-local} of \eqref{eq-main-global}.

Specifically, by subtracting the lower obstacle $\phi_1$ from the solution $u$, the problem \eqref{eq-main-global} is reduced to the double obstacle problem with the zero lower obstacle:
\begin{align}\label{eq-ori-local}\tag{$FB_{local}$}
\begin{split}
F(D^2 u,x) =f\chi_{\{0< u< \psi\} } + F(D^2\psi,x) \chi_{\{0<u=\psi\}}, \quad 0\le u\le \psi & \quad \text { in } B_1,
\end{split}
\end{align}
with $\psi\in C^{1,1}(\overline{B_1}) \cap C^{2,1}(\overline{\{\psi>0\}})$, $f\in C^{0,1}(B_1)$, see Subsection \ref{rem-red} for more detail. Furthermore, we consider a general problem \eqref{eq-main-local} of \eqref{eq-ori-local}, which relaxes the sign condition of $u$ in \eqref{eq-ori-local} (i.e., the solution could be below the lower zero obstacle):
\begin{align}\label{eq-main-local}\tag{$FB_{\textit{nosign local}}$}
\begin{split}
F(D^2 u,x) =f\chi_{\Omega(u) \cap\{ u< \psi\} } + F(D^2\psi,x) \chi_{\Omega(u)\cap \{u=\psi\}},\quad u\le \psi \quad &\text { in } B_1,
\end{split}
\end{align}
where
\begin{align*}
\Omega(u):=B_1 \setminus (\{u= 0\} \cap \{\nabla u=0\}) \quad \text{and}\quad f\in C^{0,1}(B_1),
\end{align*}
with \emph{ the upper obstacle function} 
$$\psi \in C^{1,1}(B_1) \cap C^{2,1}(\overline{\Omega(\psi)}),\quad \Omega(\psi):=B_1 \setminus (\{\psi= 0\} \cap \{\nabla \psi=0\}).$$
By obtaining the regularity of the free boundary $\Gamma(u) := \partial \Omega(u) \cap B_1$ of \eqref{eq-main-local} Theorem \ref{thm-main-2}, we have the regularity of the free boundary $\partial \{ u =\psi_1\}$ for \eqref{eq-main-global} as a corollary, see Corollary \ref{cor-main}.

The result of the problem \eqref{eq-main-local} is a generalization of the theory for the no-sign single obstacle problem ($\psi = \infty$ in \eqref{eq-main-local}) studied in \cite{CKS, FS14}. Moreover, it is an extension for the result of the problem for Laplacian in \cite{LPS}.

\subsection{Methodology and contents}\label{method plan}

The main idea to have the regularity of the free boundary, $\Gamma(u)$ of \eqref{eq-main-local}, which corresponds to $\partial \{u=\psi_1\}$ in \eqref{eq-main-global}, is considering the upper obstacle $\psi$ as a solution of the single obstacle problem, \eqref{eq-ori-local} with $\psi=\infty$. Additionally, we apply the method of blowup to the upper obstacle $\psi$ with the thickness assumption of the zero set of $\psi$ at $0$, which means that the zero set near the free boundary point $0$ is sufficiently large in some sense, see Subsection \ref{subsec-def}. Then, the blowup $\psi_0$ of the upper obstacle $\psi$ is of the form $c (x_n^+)^2$, $c>0$ and it is crucially used to have the regularity of the free boundary.

The main difficulty to have the regularity of the free boundary is the lack of monotonicity formulas, used in the problem for Laplacian in \cite{LPS}. Precisely, in the paper, by using the formulas, we have the \emph{classification of global solutions}, a global solution of \eqref{eq-main-local} in whole domain $\re^n$ is of the form $c (x_n^+)^2$, $c>0$. However, it is not applicable for the fully nonlinear case due to the nonlinearity.

Hence, for the fully nonlinear operator, we focus on the fact that the global solution $u$ is zero in a half-space $\{x_n\le 0\}$. Then, the optimal ($C^{1,1}$) regularity for $u$ implies that $\partial_e u/x_n$ is finite in $\re^n$. Therefore, we prove that $\partial_e u/x_n$ is identically zero in $\re^n$ for any direction $e \in \mathbb{S}^{n-1}\cap e_n^{\perp}$, which implies that $u$ is a one-dimensional function and it is of the form $\frac{c}{2}(x_n^+)^2$, for a positive constant $c$. It is noticeable that similar arguments for the second derivative have been introduced in \cite{LS01}, and the one for the first derivative as above has been considered in \cite{IM16b} in the study of the free boundary near the fixed boundary.

Now, we summarize the contents of this paper. In Subsection \ref{sec-exi}, we have the existence and uniqueness of the $W^{2,p}$ ($n<p< \infty$) solution of \eqref{eq-main-global} by using a penalization method, Proposition \ref{W2p}. Since the obstacles $\phi_1$ and $\phi_2$ have $C^{1,1}$ regularity, we consider the penalization method with bounded penalty term. In Subsection \ref{sec-opt}, we have the optimal regularity of the solution of \eqref{eq-main-global} Proposition \ref{prop-optimal-est} by using the quadratic growth of the solution of \eqref{eq-ori-local}, Proposition \ref{prop-quad-grow}. 

 In Subsection \ref{sec-cla}, we obtain the classification of global solutions Proposition \ref{cla}, the global solution $u$ to \eqref{eq-main-local} with the upper obstacle $\psi=c (x^+_n)^2$ is also of the form $c_1 (x_n^+)^2$, for some $c, c_1 >0$, by using the argument discussed in the previous paragraph. 
 
Therefore, in Subsection \ref{sec-dir}, we prove the directional monotonicity and the proof of the regularity for the free boundaries, Theorem \ref{thm-main-2}, using the methods considered in \cite{Lee98, PSU, IM16a, LPS} and references therein.

\begin{remark}
The reason to set the regularity of the obstacle functions $\phi_1$, $\phi_2$ in \eqref{eq-main-global}, and $\psi$ in \eqref{eq-ori-local}, \eqref{eq-main-local} to $C^{1,1}$ is closely related to the main idea introduced in the first paragraph of the previous subsection. Indeed, to apply the method of blowup to the upper obstacle $\psi$, the regularity of $\psi$ should be at least $C^{1,1}$. Furthermore, the thickness assumption of the zero set of $\psi$ means that the region where the equation $F(D^2\psi) =0$ satisfies ($\{ \psi=0\}$) is sufficiently large. Hence, $D^2 \psi$ should not be continuous, and therefore, the regularity of the upper obstacle $\psi$ should not be better than $C^{1,1}$.
\end{remark}

\subsection{Reduction of \eqref{eq-main-global}}\label{rem-red}
By subtracting the lower obstacle $\phi_1$ from the solution $u$, we reduce the problem \eqref{eq-main-global} to the double obstacle problem with zero lower obstacle. Specifically, we define $\tilde F(\mathcal M,x):= F(\mathcal M+D^2\phi_1,x)-F(D^2\phi_1,x)$ and $v:=u-\phi_1$, where $u$ is a $W^{2,p}$ ($n<p< \infty$) solution of \eqref{eq-main-global}. Then,
\begin{align*}
\tilde F(D^2v,x)&=F(D^2u,x)-F(D^2\phi_1,x)\\
&=-F(D^2 \phi_1,x) \chi_{\{\phi_1<u< \phi_2\}}+\left(F(D^2 \phi_2,x)-F(D^2 \phi_1,x)\right) \chi_{\{\phi_1<u=\phi_2\}}\\
&=-F(D^2 \phi_1,x) \chi_{\{0<v< \phi_2-\phi_1\}}+\tilde F(D^2 (\phi_2-\phi_1),x) \chi_{\{0<v=\phi_2-\phi_1\}}.
\end{align*}
By replacing $f=-F(D^2 \phi_1,x)$, $\psi=\phi_2-\phi_1$ and reusing $v=u-\phi_1$ by $u$, $\tilde F(\mathcal M,x)= F(\mathcal M+D^2\phi_1,x)-F(D^2\phi_1,x)$ by $F(\mathcal M,x)$, $u$ is a $W^{2,p}$ solution of 
\begin{equation}\label{rdop}
F(D^2 u,x) =f\chi_{\{0< u< \psi\} } + F(D^2\psi,x) \chi_{\{0<u=\psi\}}\quad \text {a.e. in } D,
\end{equation}
with $0\le u\le \psi \text { in } D$, $f\in L^{\infty}(D)$, and $\psi \in C^{1,1}(\overline{D}).$ Since we discuss the local regularity of the free boundaries, we consider a local form \eqref{eq-ori-local} of \eqref{rdop}.

\subsection{Notations}
We will use the following notations throughout the paper. 
$$\begin{array}{ll}
C, C_0, C_1 &\hbox{generic constants }\cr 
\chi_E &\hbox{the characteristic function of the set } E, (E \subset \Rn)\cr 
\overline E &\hbox{the closure of } E\cr 
\partial E &\hbox{the boundary of a set } E \cr
|E| &n\hbox{-dimensional Lebesgue measure of the set } E\\ 
B_r(x), B_r \qquad &\{y\in \Rn: |y - x|<r\}, \quad B_r(0) \\
\Omega(u), \Omega(\psi) & \hbox{see Equation } \eqref{eq-main-local}\\
\Lambda(u), \Lambda(\psi) & B_1\setminus \Omega(u), B_1\setminus \Omega(\psi) \\
\Gamma(u),\Gamma^\psi(u) & \partial \Lambda(u)\cap B_1, \partial \{u=\psi\} \cap B_1\\
\Gamma^d(u) & \Gamma(u)\cap \Gamma^\psi(u) \quad \text{(the intersection of free boundaries)}\\ 
\partial_{\mathbf{\nu}}, \partial_{\nu e}& \hbox{first and second directional derivatives }\\
P_r(M), P_\infty (M) & \hbox{see Definition } \ref{loc sol}, \ref{glo sol}\cr
\delta_r(u,x), \delta_r(u) &\hbox{see Definition } \ref{thi}\\
\mathcal P^{+}, \mathcal P^{-}& \text{Pucci operators}\\
\mathcal S, \overline {\mathcal S},\underline {\mathcal S},\mathcal S^* & \text{the viscosity solution spaces for the Pucci operators}
\end{array}
$$ 

We refer to the book of Caffarelli-Cabr\'{e} \cite{CC}, for the definitions of the viscosity solution, Pucci operators $\mathcal P^{\pm}$ and the spaces of viscosity solutions of the Pucci operators $\mathcal S(\lambda_0, \lambda_1,f)$, $\overline{\mathcal S}(\lambda_0, \lambda_1,f)$, $\underline{\mathcal S}(\lambda_0, \lambda_1,f)$, and $\mathcal S^*(\lambda_0, \lambda_1,f)$.

\subsection{Conditions on $F=F(\mathcal M,x)$} \label{def ful}
We assume that the fully nonlinear operator $F(\mathcal M,x)$ satisfies the following conditions:
\begin{enumerate}[(F1)]
\item $F(0,x)=0$ \quad for all $x\in \re^n$.
\item $F$ is uniformly elliptic with ellipticity constants $0<\lambda_0 \le \lambda_1< +\infty$, that is
$$\lambda_0 \norm{\mathcal N} \le F(\mathcal M+\mathcal N,x)-F(\mathcal M,x) \le \lambda_1 \norm{\mathcal N},$$
for any symmetric $n \times n$ matrix $\mathcal M$ and positive definite symmetric $n \times n$ matrix $\mathcal N$.
\item $F(\mathcal M,x)$ is convex in $\mathcal M$ for all $x\in \re^n$.
\item $$|F(\mathcal M,x)-F(\mathcal M,y)| \le C\norm{\mathcal M} |x-y|^\alpha,$$ 
for some $0<\alpha\le 1$.
\end{enumerate}
\begin{enumerate}[(F4)']
\item $$|F(\mathcal M,x)-F(\mathcal M,y)| \le C(\norm{\mathcal M}+C_1) |x-y|^\alpha, \quad \text{ for some } 0<\alpha\le 1.$$ 
\end{enumerate}
%Furthermore, we introduce the \emph{Pucci operators} $\cP^{\pm}$, with $0< \lambda_0\le \lambda_1< +\infty$ as
%$$\cP^-(\mathcal M,\lambda_0,\lambda_1):=\inf_{\lambda_0 Id \le\mathcal N \le \lambda_1 Id} Tr \mathcal N \mathcal M, \quad \cP^+(\mathcal M,\lambda_0,\lambda_1):=\sup_{\lambda_0 Id \le \mathcal N \le \lambda_1 Id} Tr\mathcal N \mathcal M,$$
%for any symmetric $n \times n$ matrix $\mathcal M$.
\begin{remark}
We define oscillations of the fully nonlinear operator $F$ in the variable $x$ by 
$$\beta_F(x,x_0):=\sup_{\mathcal M\in S\setminus \{0\}} \frac{|F(\mathcal M,x)-F(\mathcal M,x_0)|}{\norm{\mathcal M}}$$
and
$$\tilde \beta_F(x,x_0) :=\sup_{\mathcal M\in S} \frac{|F(\mathcal M,x)-F(\mathcal M,x_0)|}{\norm{\mathcal M}+1}.$$
For any fixed $x_0$, the condition (F4) implies that $\beta_F$ and $\tilde \beta_F$ are $C^\alpha$ at $x_0$. Then, $\beta_F$ and $\tilde \beta_F$ satisfy the conditions for the $W^{2,p}$ and $C^{2,\alpha}$ estimates of viscosity solutions $v$ to $F(D^2v, x) = f(x)$, respectively (see Chapter 7 and 8 in \cite{CC} and \cite{Win}).

Hence, in Section \ref{sec-thm1}, we assume that $F$ satisfies (F4) and the $W^{2,p}$ estimate is used in the proof of the existence and uniqueness of $W^{2,p}$ solution \eqref{eq-main-local}, Proposition \ref{W2p} and $C^{2,\alpha}$ estimate is used in the proof of optimal regularity of solution, Proposition \ref{prop-optimal-est}.

In Section \ref{sec RFB},  we study the regularity of the free boundary for the reduced forms, \eqref{eq-ori-local} and \eqref{eq-main-local}. If $F$ is the fully nonlinear operator of \eqref{eq-main-global}, then $\tilde F= F(\mathcal M+D^2\phi_1,x)-F(D^2\phi_1,x)$, introduced in Subsection \ref{rem-red}, is the fully nonlinear operator in \eqref{eq-ori-local} and \eqref{eq-main-local}. If $F(\mathcal M,x)$ satisfies (F4), then $\tilde F(\mathcal M,x)$ satisfies (F4)' and $\tilde \beta_{\tilde F}(x,x_0)$ is $C^\alpha$ for the variable $x$ at fixed $x_0 \in \re^n$. Hence, we have the $C^{2,\alpha}$ estimate of viscosity solutions $v$ to $\tilde F(D^2v, x) = f(x)$, and it is used in Lemma \ref{blo1}, to have that the blowup $u_0$ of $u$ of \eqref{eq-main-local} is a global solution.

We note that, in Section \ref{sec RFB}, when we study the regularity of the free boundary for the reduced forms, \eqref{eq-ori-local} and \eqref{eq-main-local}, we denote the fully nonlinear operator by $F$, instead of $\tilde F$. Hence, we assume (F4)' for a fully nonlinear operator $F$, in Section \ref{sec RFB}.
\end{remark}

%and in Theorem \ref{thm-main-2}, we assume (F4)' instead of (F4). 

% which are essential theories to study on obstacle problems.
%The theory for estimates should be discussed in advance, before we research on the double obstacle problem. 

\subsection{Definitions}\label{subsec-def} 

In this subsection, we define the rescaling, blowup, thickness of coincidence sets $\Lambda(u)$ and $\Lambda(u) \cap \Lambda(\psi)$, and solution spaces. These concepts are already discussed in the literature of the obstacle problem, e.g. \cite{Caf77, Caf98, Lee98, PSU, FS14, LPS}. We introduce the concepts for \eqref{eq-main-local}, for the reader's convenience.

In order to find the possible configuration of the solution near the free boundary, the following blowup concept has been used heavily at \cite{Caf77, Fri} and other references. For a $W^{2,n}$ solution, $u $, of \eqref{eq-main-local} in $B_r$, we define the \emph{rescaling function} of $u$ at $x_0\in \partial \Lambda(u) \cap B_r$ with $\rho>0$ as 
$$u_\rho(x)=u_{\rho,x_0}(x):=\dfrac{u(x_0+\rho x)-u(x_0)}{\rho^2}, \quad \text{ for } x\in (B_{r}-x_0)/ \rho.$$
By optimal ($C^{1,1}$) regularity of solution $u$ (Theorem \ref{thm-main-1}), for any sequence $\rho_i \to 0$, there exists a subsequence $\rho_{i_j}$ of $\rho_i$ and $u_0 \in C^{1,1}_{loc}(\re^n)$ such that 
$$u_{\rho_{i_j}} \rightarrow u_0 \text{ uniformly in } C^{1,\alpha}_{loc}(\R^n) \quad \text{ for any } 0<\alpha<1.$$ The limit function $u_0$ is a \emph{blowup of $u$ at $x_0$}.

% Let $u$ be a solution of \eqref{eq-main-local} in $\re^n$. Then, for a sequence $\lambda_i \to \infty$, there exists a subsequence $\lambda_{i_j}$ of $\lambda_i$ and $u_0 \in C^{1,1}_{loc}(\re^n)$
%such that
%$$u_{\lambda_{i_j}} \rightarrow u_\infty \text{ in } C^{1,\alpha}_{loc}(\R^n) \quad \text{ for any } 0<\alpha<1.$$ Such $u_\infty$ is called a \emph{shrink-down of $u$ at $x_0$}.

\begin{definition}\label{thi}(Thickness of the coincident set $\Lambda(u)$ ) We denote by $\delta_r(u,x)$ the thickness of $\Lambda(u)$ in $B_r(x)$, i.e.,
$$\delta_r(u,x):=\dfrac{\text{MD}(\Lambda(u) \cap B_r(x))}{r},$$ 
where MD($A$), the \emph{minimal diameter} of subset $A$ of $\re^n$, is the least distance between two parallel hyperplanes containing $A$. We will use the abbreviated notation $\delta_r(u)$ for $\delta_r(u,0)$.
\end{definition}

To briefly explain the theory of the regularity of free boundary in Section \ref{sec RFB}, we define classes of local and global solutions of the problem.

\begin{definition} \label{loc sol}(Local solutions)
We say a $W^{2,n}$ function $u$ belongs to the class
$P_r(M)$ $(0<r<\infty)$, if 
$u$ satisfies
\begin{enumerate}[(i)]
\item $F(D^2 u,x) =f\chi_{\Omega(u) \cap\{u <\psi\} } + F(D^2\psi,x) \chi_{\Omega(u) \cap \{u=\psi\}}, \quad u\le \psi \quad \text { in } B_r,$ 
\item $\|D^2 u \|_{L^\infty ,B_r} \leq M$,
\item $0\in \Gamma^d(u),$
%\rred{\item [(4)] $\Lambda(\psi) \subset \Lambda(u)$}
\end{enumerate}
where $f\in C^{0,\alpha}(B_r)$ and $\psi \in C^{1,1}(B_r)\cap C^{2,\alpha}(\overline{\Omega(\psi)})$.
\end{definition}

\begin{definition} \label{glo sol}(Global solutions)
We say a $W^{2,n}$ function $u$ belongs to the class
$P_\infty(M)$, if 
$u$ satisfies
\begin{enumerate}[(i)]
\item $F(D^2 u) =\chi_{\Omega(u) \cap\{u <\psi\} } + F(D^2\psi) \chi_{\Omega(u) \cap \{u=\psi\}}, \quad u\le \psi \quad \text { in } \re^n,$
\item $F(D^2 \psi) =a \chi_{\Omega(\psi)}$ \text{ in } $\re^n$, for a constant $a>1$,
\item $\| D^2 u \|_{\infty,\re^n} \leq M$,
\item $0\in \Gamma(u).$
\end{enumerate}
\end{definition}

%\subsection{Main Results}

\subsection{Main Theorems}
The purposes of the paper are to obtain the existence, uniqueness, and optimal regularity of the solution for the double obstacle problem and the regularity of the free boundary. The main theorems are as follows: 

\begin{theorem}[Existence, uniqueness and optimal regularity]\label{thm-main-1}
Assume $F$ satisfies (F1)-(F4). Then the following holds:
\begin{enumerate}[(i)]
\item For $n<p< \infty$, there exist $W^{2,p}$ solution $u$ of \eqref{eq-main-global} with $\phi_1, \phi_2\in C^{1,1}(\overline{D})$, $\partial D \in C^{2,\alpha}$, $g\in C^{2,\alpha}(\overline{D})$, and $\phi_1\le g \le \phi_2$ in $D$.
\item For any compact set $K$ in $D$, we have 
$$\| u\|_{ C^{1,1}(K)}\leq M<\infty,$$
for some constant $M=M(\|u\|_{ L^{\infty}(D)}, \| \phi_1\|_{ C^{1,1}(D)},\|\phi_2\|_{ C^{1,1}(D)}, dist(K, \partial D))>0$.
\end{enumerate}
\end{theorem}

\begin{theorem}[Regularity of free boundary] \label{thm-main-2}
Assume $F\in C^1$ satisfies (F1)-(F3) and (F4)' and let $u \in P_1(M)$ with an upper obstacle $\psi$ such that 
$$0\in \partial \Omega(\psi), \quad \lim_{x \to 0, x\in \Omega(\psi)} F(D^2\psi(x),x)>f(0), \quad f\ge c_0>0 \text{ in } B_1,$$
and 
$$ \inf{\left\{F(D^2\psi,x), F(D^2 \psi,x)-f\right\}} \ge c_0>0 \text { in }\Omega(\psi).$$
Suppose
\begin{equation}\label{thi psi zero}
\delta_r(u, \psi):=\dfrac{\text{MD}(\Lambda(u)\cap \Lambda(\psi) \cap B_r)}{r} \ge \epsilon_0 \quad \text{ for all } r< 1/4.
\end{equation}
Then there is $r_0=r_0(u, c_0, \norm{\nabla F}_{L^\infty(B_{M+\|\phi_1\|_{ C^{1,1}(D)}}\times B_1)}, \norm{F}_{L^\infty(B_M+\|\phi_1\|_{ C^{1,1}(D)} \times B_1)},\norm{\nabla f}_{L^\infty(B_1)})>0$, such that $\Gamma(u) \cap B_{r_0}$ is a $C^1$ graph. %, where $M$ is in Theorem \ref{thm-main-1} with $B_1$ and $D$.
\end{theorem}

Since Theorem \ref{thm-main-2} is for the reduced forms \eqref{eq-ori-local} and \eqref{eq-main-local} with $\tilde F(M,x)=F(M+D^2 \psi, x)-F(D^2 \psi, x)$, where $F$ is the fully nonlinear operator in \eqref{eq-main-global}, the local regularity of the free boundary for \eqref{eq-main-global} is obtained as a corollary.

\begin{cor}\label{cor-main}
Let $u$, $\phi_1$, and $\phi_2$ be as in Theorem \ref{thm-main-1} and we assume that $\phi_2-\phi_1 \in C^{2, 1}(\overline{ \{\phi_1<\phi_2\}})$ and $\phi_1 \in C^{2, 1}(D)$. Suppose that $0\in \partial \{u>\phi_1\} \cap \partial\{u<\phi_2\}$,
$$0\in \partial \{\phi_1<\phi_2\},\quad \lim_{x \to 0, x\in \{\phi_1<\phi_2\}} F(D^2\phi_2(x),x)>0,\quad -F(D^2\phi_1,x)\ge c_0>0 \text{ in } B_1, $$
and
$$ \inf{\left\{F(D^2\phi_2,x)-F(D^2\phi_1,x), F(D^2 \phi_2,x)\right\}} \ge c_0>0 \text { in }\{\phi_1<\phi_2\}.$$
Suppose
\begin{equation*}
\delta_r(\phi_2-\phi_1, z)\ge \epsilon_0 \quad \text{ for all } r< 1/4, z \in \partial\{\phi_1<\phi_2\}
\end{equation*}
and
\begin{equation*}
\delta_r(u-\phi_1, \phi_2-\phi_1)\ge \epsilon_0 \quad \text{ for all } r< 1/4.
\end{equation*}
Then, there is $r_0=r_0(v-\phi_1, c_0, \norm{\nabla F}_{L^\infty(B_M\times B_1)}, \norm{F}_{L^\infty(B_M\times B_1)},\norm{\nabla F(D^2 \phi_1(x),x)}_{L^\infty(B_1)})>0$ such that $\partial\{ u=\phi_1\} \cap B_{r_0}$ is a $C^1$ graph.
\end{cor}

\begin{remark}
We assume the thickness of $\Lambda(\psi)$ and $\Lambda(u)$ satisfies the

 assumption \eqref{thi psi zero} in Theorem \ref{thm-main-2}. Then, the assumption implies that
\begin{equation}\label{thi psi0 zero0}
\delta_r(u_0, \psi_0) \ge \epsilon_0 \quad \text{ for all } r>0,
\end{equation}
for any blowups $u_0$ and $\psi_0$ of $u$ and $\psi$ at $0$, respectively. By \eqref{thi psi0 zero0}, we have that the blowups $\psi_0$ of $\psi$ are the half-space type upper obstacle, $\psi_0=\frac{a}{2} (x^+_n)^2$, in an appropriate system of coordinates, see e.g. Proposition 4.7 of \cite{LP}. Furthermore, \eqref{thi psi0 zero0} implies that the blowup $u_0$ of $u\in P_1(M)$ is a nonnegative function and moreover, $u$ is also nonnegative in a neighborhood of $0$, see Section \ref{sec-cla} and \ref{sec-dir}. Thus, the solution $u\in P_1(M)$ to the general problem \eqref{eq-main-local} becomes a solution of \eqref{eq-ori-local}.

In contrast to the Laplacian case in \cite{LPS}, we have the regularity of one of the two free boundaries. Precisely, we only have the regularity of the free boundary $\partial\{u=\phi_1\}$ which is emerged by the  lower coincident set $\{u=\phi_1\}$, not $\partial\{u=\phi_2\}$, for \eqref{eq-main-global} in Corollary \ref{cor-main}. The free boundary $\partial\{u=\phi_1\}$ is corresponding to  $\Gamma(u)$ in Theorem \ref{thm-main-2}, for the general problem \eqref{eq-main-local}. The reason for this result is the lack of the regularity theory for the free boundary of "single" obstacle problem for "concave" fully nonlinear operator, see Remark \ref{rere1}.
\end{remark}

%Suppose
%\begin{equation}\label{uni thi psi}
%\delta_r(\psi, z)\ge \epsilon_0 \quad \text{ for all } r< 1/4, z \in \Gamma(\psi)
%\end{equation}

\section{Existence, Uniqueness and Optimal Regularity}\label{sec-thm1}

%In this section, we discuss existence and optimal regularity of the solution of (FB). The existence of the solution of 'single' obstacle problem and the double obstacle problem (FB) with $C^2$ obstacles by using the penalization method were already considered in \cite{Fri, Lee98, Duq}. The $C^2$ regularity for the obstacles is necessary in the penalization method, since the fact that the Hessian matrix at local minimum of the $C^2$ function is semi-positive definite matrix is used in the theory, see e.g. \eqref{bou}. However, we study the double obstacle problem with non-smooth obstacles $\phi_1$ and $\phi_2$ in this paper. Hence, we apply the penalization method to approximated problem with the mollified obstacles $\phi_1^\epsilon$, $\phi_2^\epsilon$ and pass to a limit as $\epsilon \to 0$ to obtain the solution for the original problem (FB) with non-smooth obstacles $\phi_1$ and $\phi_2$, see Proposition \ref{prop-W2p}.

\subsection{Existence, uniqueness of $W^{2,p}$ solution}\label{sec-exi}

For the single obstacle problem in \cite{Fri, Lee98}, the authors used an unbounded penalization term $\beta_\epsilon(z)$, such that $\beta_\epsilon(z)$ to $- \infty$, for $z<0$, $\epsilon \to 0$. Then, $C^2$ regularity for obstacle function $\phi$ is needed to show that $\beta_\epsilon(u_\epsilon-\phi)$ is bounded, where $u_\epsilon$ is a solution of the penalization problem for the single obstacle problem with the obstacle function $\phi$. 

On the other hand, in this subsection, we consider a penalization problem \eqref{penal} with a new penalty term $\beta_\epsilon$, whose $L^\infty$ norms are uniformly bounded by a constant that depends only on $C^{1,1}$ norms of the obstacle functions $\phi_1$ and $\phi_2$. 

Then, we have solutions $u_\epsilon$ of the penalization problem \eqref{penal} such that $W^{2,\infty}$ norms of $u_\epsilon$ are uniformly bounded. Hence, there is a limit function $u_0$ of $u_\epsilon$ as $\epsilon \to 0$ in $W^{2,p}$ sense. Finally, we prove that the limit function $u_0$ of $u_\epsilon$ is the unique solution of \eqref{eq-main-global} with the obstacle functions $\phi_1\in C^{1,1}$ and $\phi_2 \in C^{1,1}$.

\begin{prop}\label{W2p}
Assume $F$ satisfies (F1)-(F4). For $n<p <\infty$, there is a unique viscosity solution $u \in W^{2,p}(D)$ of (FB) with
$$\norm{u}_{W^{2,p}(D)}\le C\left(\norm{F(D^2 \phi_1,x)}_{L^\infty(D)},\norm{F(D^2 \phi_2,x)}_{L^\infty(D)}\right),$$
where $\phi_1,\phi_2 \in C^{1,1}(\overline D), \partial D \in C^{2,\alpha}$, $g^{2,\alpha} \in C(\overline{D})$, and $\phi_1\le g \le \phi_2$ on $\partial D$.\
\end{prop}

\begin{proof}
Let $\beta_1(z)\in C^\infty(\R)$ be a function satisfying
\begin{displaymath}
\left\{\begin{array}{l l}
%\beta'_\epsilon(z) \ge 0& \text{ for } z\in \re,\\
\beta_1(z)= -\max\left\{\norm{F(D^2 \phi_1,x)}_{L^\infty(D)}, \norm{F(D^2 \phi_2,x)}_{L^\infty(D)}\right\} & \text{ if } z<-1,\\
\beta_1(z)= 0 & \text{ if } z>1,\\
\beta_1(z)\le 0 & \text{ in } z \in \re,
\end{array} \right.
\end{displaymath}
and define $\beta_\epsilon(z):=\beta_1(z/\epsilon)$, for $\epsilon>0$. We consider a penalization problem,
\begin{equation}\label{penal}
\left\{ \begin{array}{l l}
F(D^2u,x) =\beta_\epsilon (u-\phi_1)-\beta_\epsilon (\phi_2-u) & \quad \text{ in } D,\\
u(x)= g(x) & \quad\text{ on } \partial D.\\
\end{array}\right.
\end{equation}
By the $W^{2,p}$ regularity in \cite{CC} and \cite{Win}, for each $v\in C^{0,\alpha}(D)$ $(0<\alpha<1)$ there is a unique solution $w \in W^{2,p}(D)$ ($n<p< \infty)$, of 
\begin{equation*}
\left\{ \begin{array}{l l}
F(D^2w,x)=\beta_{\epsilon} (v-\phi_1)-\beta_{\epsilon} (\phi_2-v) & \quad \text{ in } D,\\
w(x)= g(x) & \quad \text{ on } \partial D,\\
\end{array}\right.
\end{equation*}
with 
\begin{equation*}
\norm{w}_{W^{2,p}(D)}\le \norm{w}_{L^{\infty}(D)} +\norm{g}_{W^{2,p}(D)}+\norm{\beta_{\epsilon} (v-\phi_1)-\beta_{\epsilon} (\phi_2-v)}_{L^p(D)},
\end{equation*}
By the boundedness of $\beta_{\epsilon}$, we have 

\begin{equation}\label{w2p w}
\norm{w}_{W^{2,p}(D)}\le C_0,
\end{equation}
where $C_0$ is a constant which is independent for $\epsilon$ and $v$. 

Let us consider a map $S$ such that $w=Sv$ for $v\in C^{0, \alpha}(D)$. Since $W^{2,p}$ space is compactly embedded in $C^{0,\alpha}$, the boundedness of $W^{2,p}$ norm of $w$, \eqref{w2p w} implies that $S|_{B_{C_0}} :B_{C_0} \to B_{C_0}$ is a compact map, where $B_{C_0}$ is the $C_0$ ball centered at $0$ in $C^{0,\alpha}(D)$ and $S|_{B_{C_0}}$ is the function $S$ from $B_{C_0}$ to $B_{C_0}$ defined by $S|_{B_{C_0}} (v) = S(v)$. Furthermore, the $W^{2,p}$ estimate implies that $S|_{B_{C_0}}$ is continuous. Hence, by Schauder's fixed-point theorem, there is a function $u_\epsilon \in B_{C_0}$ such that $S|_{B_{C_0}} u_\epsilon = u_\epsilon$, i.e., there is $u_{\epsilon}\in W^{2,p}(D)$ such that $u_{\epsilon}$ is a solution of \eqref{penal} and $\norm{u_{\epsilon}}_{W^{2,p}(D)}\le C_0$, where $C_0$ does not depend on $\epsilon$. Then, there is a sequence $\epsilon=\epsilon_i \to 0$ and $u \in W^{2,p}(D)$ such that
$$u_\epsilon \to u \quad \text{ weakly in } W^{2,p}(D), \quad n<p < \infty.$$
Thus, we have that $\norm{u}_{W^{2,p}(D)}\le C_0$ and
$$u_\epsilon \to u \quad \text{ uniformly in } D.$$

We claim that $u$ is a solution of the double obstacle problem \eqref{eq-main-global}. First, we are going to prove that $F(D^2 u,x) \ge 0$ in $\{u> \phi_1\} \cap D$. Let $x_0$ be a point in $\{u> \phi_1\} \cap D$ and let $\delta= (u(x_0)-\phi_1(x_0))/2$. Then, by the uniform convergence of $u_\epsilon$ to $u$, there is a ball $B_r(x_0) \Subset \{u> \phi_1\} \cap D$ and $\epsilon_0>0$ such that $u_\epsilon-\phi_1 \ge \delta$ in $B_r(x_0)$, for $\epsilon<\epsilon_0$. By the definition of $\beta_\epsilon$, for $\epsilon \le \min \{ \epsilon_0, \delta\}$, we have 
$$\beta_\epsilon(u_\epsilon-\phi_1) \equiv0 \quad \text{ and } \quad F(D^2 u_\epsilon,x) \ge 0\text{ in } B_r(x_0).$$
By the closedness of the family of viscosity solutions, Proposition 2.9 of \cite{CC}, the uniform convergence of $u_\epsilon$ to $u$ implies that $F(D^2 u,x) \ge 0\text{ in } B_r(x_0)$. Since $x_0\in \{u>\phi_1\} \cap D$ is arbitrary, we obtain $F(D^2 u,x) \ge 0$ in $\{u> \phi_1\} \cap D$. We also have $F(D^2 u,x) \le 0$ in $\{u< \phi_2\} \cap D$, from the same argument as above. 

Next, we prove that $\phi_1 \le u \le \phi_2$ in $D$. Suppose that $\{u<\phi_1\} \cap D$ is not empty and let $x_0$ be a point in $ \{ u< \phi_1\} \cap D$ . Then, by the uniform convergence of $u_\epsilon$, there is a ball $B_r(x_0)$ such that 
$$\beta_\epsilon(u_\epsilon-\phi_1)= -\max\left\{\norm{F(D^2 \phi_1,x)}_{L^\infty(D)}, \norm{F(D^2 \phi_2,x)}_{L^\infty(D)}\right \}, \beta_\epsilon(\phi_2-u_\epsilon) \equiv 0 \text{ in } B_r(x_0)$$ and
$$F(D^2u_\epsilon,x) \le F(D^2 \phi_1) \text{ in } B_r(x_0), \text{ for sufficiently small } \epsilon. $$ 
Consequently, $F(D^2u,x) \le F(D^2 \phi_1) \text{ in } \{u< \phi_1\}\cap D$. Moreover, the boundary condition $\psi_1\le u=g$ on $\partial D$ implies $\{u< \phi_1\}\cap D \Subset D$ and $u \equiv \phi_1$ on $\partial (\{ u< \phi_1\} \cap D )$. Hence, by the maximum principle, we have $u \ge \phi_1$ in $\{u< \phi_1\}\cap D$ and it is a contradiction. The same method implies that $\{u> \phi_2\}\cap D= \emptyset$ and $\phi_1 \le u \le \phi_2$ in $D$. Hence, $u$ is a solution of \eqref{eq-main-global}.

In order to prove the uniqueness, we suppose that there are two solutions $u_1$ and $u_2$ of \eqref{eq-main-global} and $\{u_1<u_2\} \cap D$ is not empty. In $\{u_1<u_2\} \cap D$, the conditions $\phi_1 \le u_1 \le \phi_2$ and $\phi_1 \le u_2 \le \phi_2$ in $D$ imply that $\phi_2>u_1$ and $u_2> \phi_1$ and we have $F(D^2u_1,x) \le 0 \le F(D^2u_2,x)$ in $\{u_1<u_2\} \cap D$. Furthermore, by the boundary condition for $u_1$ and $u_2$, we have that $u_1\equiv u_2$ on $\partial (\{u_1<u_2\} \cap D)$. Therefore, by the comparison principle, we have that $u_1\ge u_2$ in $\{u_1<u_2\} \subset D$ and we arrive at a contradiction.
\end{proof}

\subsection{Optimal Regularity}\label{sec-opt}

In this subsection, we prove the optimal regularity of the double obstacle problem \eqref{eq-main-global} with $C^{1,1}$ obstacles by using the reduced form of \eqref{eq-ori-local}. We will first prove the quadratic growth of the solution at the free boundary point. 

\begin{definition}\label{def-setofsol}
For a positive constant $c'$, let $\mathcal G(c')$ be a class of solutions $u \in W^{2,n}(B_1)$ of
\begin{equation}\label{eqeq}
F(D^2 u,x)=f(x)\chi_{\{0< u< \psi\}}+F(D^2 \psi,x)\chi_{\{0< u=\psi\}},\quad 0 \le u \le \psi \text{ in } B_1,
\end{equation}
with $|f(x)|, |F(D^2 \psi,x)|, |\psi| \le c' \text{ in } B_1 $ and $0\in \Gamma(u)$.
\end{definition}

\begin{prop}[Quadratic growth]\label{prop-quad-grow}
Assume $F$ satisfies (F1) and (F2). For any $u\in \mathcal G(c')$, we have
\begin{equation}\label{eq QG}
S(r,u):=\sup_{x\in B_r} u(x) \le C_0 r^2,
\end{equation}
for a positive constant $C_0=C_0(c')$.
\end{prop}

\begin{proof}
First, we show that there is a positive constant $C_0$ such that 
\begin{equation}\label{2u0}
S(2^{-j-1},u) \le max (C_02^{-2j},2^{-2} S(2^{-j},u)) \quad \text{ for all } j \in \N\cup\{0, -1\}.
\end{equation}
Suppose it fails, then, for each $j \in \N\cup\{0, -1\}$, there exists $u_j \in \mathcal G$ such that
\begin{equation}\label{2u}
S(2^{-j-1},u_j) > max (j2^{-2j},2^{-2} S(2^{-j},u_j)).
\end{equation}

We consider
$$\tilde u_j (x) :=\frac{u(2^{-j}x)}{S(2^{-j-1},u)} \quad x \in B_{2^j}.$$
Then, by the definition of $\tilde u$ and $\eqref{2u}$,
$$S(\tilde u_j,1/2)=1, \quad S(\tilde u_j,1)=4, \quad \text{ and } \quad \tilde u_j(0)=0.$$
%\begin{align*}
%S(\tilde u_j,1/2)&=1,\\
%S(\tilde u_j,1)&=4,\\
%\tilde u_j(0)&=0.
%\end{align*}

Since $u \in \mathcal G (c')$, by the condition (F1) and Proposition 2.13 of \cite{CC}, we know that $u \in \mathcal S^*(\frac{\lambda}{n}, \Lambda, c').$ Thus, the inequality \eqref{2u} implies 
$$\mathcal P^{+}(D^2 \tilde u(x))=\frac{2^{-2j}}{S(2^{-j-1},u)} \cdot \mathcal P^{+}(D^2 u(2^{-j}x)) \ge -\frac{c'}{j}$$
and
$$\mathcal P^{-}(D^2 \tilde u(x))=\frac{2^{-2j}}{S(2^{-j-1},u)} \cdot \mathcal P^{-}(D^2 u(2^{-j}x)) \le \frac{c'}{j},$$
where $P^{\pm}$ are Pucci operators, i.e., we obtain that $\tilde u \in \mathcal S^*(\lambda/n, \Lambda, c'/j).$ By Harnack inequality (Theorem 4.3 of \cite{CC}) and $C^{\alpha}$ regularity (Proposition 4.10 of \cite{CC}), we know that $\tilde u_j \to \tilde u$ in ${B_1}$, up to subsequence and
$$\tilde u \in \mathcal S^*(\lambda/n, \Lambda,0) \quad \text{ in } B_1,$$ 
$\tilde u\neq 0$ in $B_{1/2}$, and $\tilde u(0)=0.$ In other words, a nontrivial viscosity solution $\tilde u$ has its minimum at an interior point. Hence, by the strong maximum principle, it is a contradiction.

Next, we claim that 
\begin{equation}\label{2u2}
S(2^{-j},u) \le C_02^{-2j+2}\quad \text{ for all } j \in \N\cup \{0\}.
\end{equation}
We may assume that $C_0>c'/4$. Then, \eqref{2u2} holds for $j=0$. Assume that \eqref{2u2} holds for $j=j_0$. By \eqref{2u0}, we have the inequality \eqref{2u2} for $j_0+1$,
$$S(2^{-(j_0+1)},u) \le max (C_02^{-2j_0},2^{-2} S(2^{-j_0},u))\le C_02^{-2j_0}.$$
Thus, by the mathematical induction, we have \eqref{2u2} for all $j \in \N \cup \{0\}.$

Now, take a positive number $r$, and take a natural number $j$ such that $2^{-j-1}\le r \le 2^{-j}$. Then, by \eqref{2u2}, we have
$$S(r,u) \le S(2^{-j},u) \le C_02^{-2j+2}=C_02^42^{-2j-2}\le C_02^4r^2.$$
Thus, we have the quadratic growth rate \eqref{eq QG} of $u$ at $0$.
\end{proof}

\begin{prop}[Optimal regularity]\label{prop-optimal-est} 
Assume $F$ satisfies (F1)-(F4). Let $u \in W^{2,n}(D)$ be a solution %, for any $n<p< \infty$, 
of (FB), with $\phi_1,\phi_2 \in C^{1,1}(\overline{D})$, $\partial D \in C^{2,\alpha}$, $g \in C^{2,\alpha} (\overline{D})$, and $\phi_1\le g \le \phi_2$ on $\partial D$. Then $u\in W^{2,\infty}_{loc}(D)$.

\end{prop}

\begin{proof}
Let $K$ be a compact set in $D$ and $\delta=dist(K, \partial D)$. Since $u \in W^{2,p}(D)$, %for any $n<p<\infty$, 
$D^2u=D^2 \phi_1$ a.e. on $\{u=\phi_1\}$ and $D^2u=D^2 \phi_2$ a.e. on $\{u=\phi_2\}$. Thus, it suffice to show that $\norm{u}_{W^{2,\infty}(\{\phi_1<u <\phi_2\} \cap K)}<+\infty$. Let $x_0$ be a point in $\{\phi_1<u <\phi_2\} \cap K$ and denote $d(x_0):=dist(x_0, \partial \{u=\phi_1\}\cup \partial \{u=\phi_2\})$. We may assume that $d(x_0)=dist(x_0, \partial \{u=\phi_1\})$.

%We are going to apply Proposition \ref{prop-quad-grow} to $u-\phi_1$ and $\phi_2-u$ which are solutions of the problem of the form $\eqref{eqeq}$ (see e.g. Section \ref{rem-red} and Remark \ref{rere}). 

\indent Case 1) $5d(x_0)< \delta. $

For $v:=u-\phi_1$, we have that
\begin{equation*}
\tilde F(D^2v,x)=-F(D^2 \phi_1,x) \chi_{\{0<v< \phi_2-\phi_1\}}+\tilde F(D^2 (\phi_2-\phi_1),x) \chi_{\{0<v=\phi_2-\phi_1\}} \quad \text{ in } D,
\end{equation*}
where $\tilde F(\mathcal M,x)= F(\mathcal M+D^2\phi_1,x)-F(D^2\phi_1,x)$, see Subsection \ref{rem-red}. 

Let $y_0 \in \partial B_{d(x_0)}(x_0) \cap \{u =\phi_1\}$. Then $B_{4d(x_0)}(y_0) \subset B_{5d(x_0)}(x_0) \Subset D.$ Since $\phi_1, \phi_2 \in C^{1,1}(\overline{D})$, we know that $v(4dx+y^0)/(4d)^2$ is in the solution space $\mathcal G(c')$ for a positive number $c'$ for the fully nonlinear operator $\tilde F$ which also satisfies (F1) and (F2). Then, by Proposition \ref{prop-quad-grow}, we obtain
$$\norm{u-\phi_1}_{L^{\infty}(B_{2d}(y_0))} \le C (\norm{\phi_1}_{C^{1,1}(\overline{D})},\norm{\phi_2}_{C^{1,1}(\overline{D})}) d^2.$$
Since $F(D^2u,x)=0$ in $B_{d(x_0)}(x_0) \subset \{\phi_1 <u< \phi_2 \}$, by $C^{2,\alpha}$ estimate, we have that
$$\norm{D^2(u-\phi_1)}_{L^{\infty}(B_{d/2}(x_0))}\le C\dfrac{\Vert u-\phi_1\Vert_{L^{\infty}(B_d(x_0))}}{d^2}.$$
Thus, $B_d(x_0) \subset B_{2d}(y_0)$ implies
\begin{equation*}
\norm{D^2(u-\phi_1)}_{L^{\infty}(B_{d/2}(x_0))}\le C (\norm{\phi_1}_{C^{1,1}(\overline{D})},\norm{\phi_2}_{C^{1,1}(\overline{D})}),
\end{equation*}
and
\begin{equation*}
\norm{D^2 u}_{L^{\infty}(B_{d/2}(x_0))}\le C (\norm{\phi_1}_{C^{1,1}(\overline{D})},\norm{\phi_2}_{C^{1,1}(\overline{D})}).
\end{equation*}

\indent Case 2) $5d(x_0)> \delta$.

The interior derivative estimate for $u$ in $B_{\delta/4}(x_0)$ gives
$$\norm{D^2u}_{L^{\infty}\left(B_{\delta/10}(x_0)\right)}\le C\dfrac{4^2}{\delta^2} \norm{u}_{L^{\infty}(D)}.$$

For the case $d(x_0)=dist(x_0, \partial \{u=\phi_2\})$, the same argument as above with $\phi_2-u$ implies the boundedness of the Hessian matrix of $u$. Therefore, we obtain the optimal regularity of the solution $u$ of \eqref{eq-main-global}.
\end{proof}

We note that the property for the classical single obstacle problem (i.e., \eqref{eq-main-global} with $\phi_2=\infty$ and $F=\Delta$) is obtained in \cite{Caf77, Caf98}. The growth rate for the reduced single obstacle problem (\eqref{eq-ori-local} with $\phi_2=\infty$) is discussed in \cite{LS01}, and the optimal ($p/p-1$) growth rate for the p-Laplacian case is obtained in \cite{KS, LS03}.

For the case \eqref{eq-main-local}, with $f\in C^{0,\alpha}$ and $\psi\in C^{2,\alpha}$, the optimal regularity of the solutions is obtained by the theory in \cite{IM16a} (see \cite{FS14, IM16a} and Theorem 2.1 of \cite{LPS} for more detail).

\section{Regularity of the Free Boundary $\Gamma(u)$}\label{sec RFB}
In this section, we discuss the regularity of the free boundary of the double obstacle problem, \eqref{eq-main-local}. In Subsection \ref{sec-cla}, we show the classification of the global solution, which means that the global solution $u \in P_\infty(M)$ with the upper obstacle $\psi(x)=\frac{a}{2}(x_n^+)^2$ ($a>1$) is $u=\frac{1}{2}(x^+_n)^2$ or $u=\frac{a}{2}(x_n^+)^2$, see Proposition \ref{cla}. In Subsection \ref{sec-dir}, we prove the directional monotonicity of the local solution $u\in P_1(M)$, Lemma \ref{dir mon}. Then, we have that $u\in P_1(M)$ is nonnegative in a small neighborhood $B$ of $0$, and $u$ is a solution of the simple problem \eqref{eq-ori-local} in $B$. Therefore, the blowup $u_0$ of $u$ should be $\frac{1}{2}(x^+_n)^2$, not $\frac{a}{2}(x^+_n)^2$, see Remark \ref{rere}. In other words, we have the uniqueness of the blowup $u_0=\frac{1}{2}(x^+_n)^2$ of $u\in P_1(M)$. Finally, we prove the regularity of the free boundary $\Gamma(u)$, by using the directional monotonicity.

\subsection{Non-degeneracy}
In this subsection, we study the non-degeneracy of the solution $u\in P_1(M)$, which is one of the important properties for solutions of obstacle problems. This property implies that $0$ is also on the free boundary $\Gamma(u_0)$, where $u_0$ is a blow-up of $u$ at $0\in \Gamma(u)$, and $\Gamma(u)$ has a Lebesgue measure zero.

%omit the rest of the proof since it is a repetition of the arguments for the linear case in the proof of .

%We note that the proof of the non-degeneracy is similar to the proof of the problem for Laplacian in Lemma 2.2 of \cite{LPS}.

\begin{lemma}\label{nond}
Assume $F$ satisfies (F1) and (F2). Let $u\in P_1(M)$. If $f\ge c_0 >0$ in $B_1$ and $F(D^2 \psi,x) \ge c_0 >0$ in $\Omega(\psi)$, then
\begin{equation}\label{eq nond}
\sup_{\partial B_r(x)}u \ge u(x)+\frac{c}{8\lambda_1n}r^2 \quad x\in \overline{\Omega(u)}\cap B_1,
\end{equation}
where $B_r(x) \Subset B_1$.
\end{lemma}
\begin{proof}
Let $x_0 \in \Omega(u) \cap B_1$ and $u(x_0)>0$. Since $u\le \psi$, we know that $\{u=\psi\}=\{u=\psi\}\cap \{\nabla u=\nabla \psi\}$ and therefore, $\Omega(u)\cap \{u=\psi\} \subset \Omega(\psi)$. By the assumptions for $f$ and $F(D^2 \psi,x)$, we obtain $F(D^2u,x)=f\chi_{\Omega(u) \cap \{u<\psi\}}+F(D^2 \psi,x) \chi_{\Omega(u)\cap \{u=\psi\}}\ge c_0$ in $\Omega(u)$. Thus, the uniformly ellipticity, (F2) in Definition \ref{def ful} implies
$$F(D^2 w,x) \ge F(D^2 u,x) -c \ge 0 \text{ on } B_r(x_0)\cap \Omega(u),$$ where
$$w(x):=u(x)-u(x_0)-\dfrac{c}{2\lambda_1n}|x-x_0|^2.$$
Since $w(x_0)=0$ and $w(x)<0$ on $\partial \Omega(u)$, the maximum principle on $B_r(x_0) \cap \Omega(u)$ implies
$$\sup_{\partial B_r(x_0) \cap \Omega(u)} w>0 \quad \text{ and } \quad \sup_{\partial B_r(x)}u \ge u(x)+\frac{c}{2\lambda_1n}r^2.$$

Let $x_0\in \Omega(u) \cap B_1$ and $u(x_0)\le 0$. If there is a point $x_1 \in B_{r/2}(x_0)$ such that $u(x_1) >0.$ Then by the first case in the previous paragraph for $x_1$ implies the non-degeneracy for $x_0$.

%$$\sup_{B_r(x^0)} u  \ge \sup_{\partial B_{r/2}(x^1)} u \ge u(x^1)+ \frac{c}{8n}r^2\ge u(x^0)+ \frac{c}{8n}r^2. $$
%Since $u$ is subharmonic, we have the desired inequality.
%$$\sup_{\partial B_r(x^0)} u  =\sup_{B_r(x^0)} u \ge u(x^0)+ \frac{c}{8n}r^2. $$

If $u(x) \le 0$ in $B_{r/2}(x_0)$, then $u(x) \equiv 0$ in $B_{r/2}(x_0)$ or $u(x)<0$ in $B_{r/2}(x_0)$, by the maximum principle. Since $x^0 \in \Omega(u)$, the second case is only possible and in the case, $ F(D^2 u, x) \ge c$ in $B_{r/2}(x_0)$. Then, it implies the non-degeneracy of $u$ at $x_0$.

%By using another auxiliary function $\tilde w(x):=u(x)-\frac{c |x-x_0|^2}{2n}$, we obtain 
%$$\sup_{\partial B_{r/2}(x^0)} \tilde w  = \sup_{B_{r/2}(x^0)} \tilde w  \ge  \tilde w(x^0)=u(x^0) \quad \text{ and } \quad \sup_{\partial B_{r/2}(x^0)} u \ge u(x^0)+ \frac{c}{8 n}r^2.$$
%Finally, the subharmonicity of $u$ implies the desired inequality.

For the case of $x_0\in \partial \Omega(u) \cap B_1$,we take a sequence of points $x_j \in \Omega(u)$ such that $x_j \to x_0$ as $j \to \infty$. By passing to the limit as $j$ goes to $\infty$, we have the desired inequality for $x_0\in \overline{\Omega(u)} \cap B_1.$ 
%We omit the rest of the proof since it is a repetition of the arguments for the linear case in the proof of Lemma 2.2 of \cite{LPS}.
\end{proof}
%We note that the non-degeneracy for the double obstacle problem of Laplacian is discussed in Lemma 2.2 of \cite{LPS}.

By the non-degeneracy and of the solution $u \in P_1(M)$, we have the local porosity of $\Gamma(u)$ and $\Gamma(u)$ has Lebesgue measure zero, e.g. Section 3.2.1 of \cite{PSU}.

\begin{remark}\label{rere}
For $u \in P_1(M)$, by the definition of the rescalings and blowups and the non-degeneracy, we have $\sup_{\partial B_r}u_0 \ge \frac{c}{8\lambda_1n}r^2$, for $r>0$, where $u_0$ is a blowup of $u$ at $0$. Thus, the origin $0$ is on the free boundary $\Gamma(u_0)$ of $u_0$. On the other hand, in general, we do not know that $0\in \Gamma^\psi(u)$ implies $0\in \Gamma^{\psi_0} (u_0)$, for $u \in P_1(M)$. 

However, in the same manner as the linear case in Remark 2.4 of \cite{LPS}, if we assume that $u$ is a non-negative function, then $v:=\psi-u$ is a solution of 
$$\tilde F(D^2v,x)=\left(F(D^2 \psi,x)-f\right)\chi_{\{0< v <\psi\}}+F(D^2 \psi,x) \chi_{\{0<v=\psi\}} \quad \text{ in } B_1,$$
where $\tilde F (M,x):= F(D^2\psi,x)-F(D^2 \psi-M, x)$. Then, if we assume that $F(D^2 \psi,x)-f  \ge c$ and $F(D^2 \psi,x) \ge c$ in $\{ \psi>0\}$, we have the non-degeneracy for $v$ which implies $0\in \Gamma(v_0)= \Gamma^{\psi_0} (u_0)$ and $| \Gamma(v)| = | \Gamma^\psi(u)|=0$. 
\end{remark}

%, although $v$ is not in $P_1(M)$,

\begin{remark}\label{rere1}
We also note that 
$$\tilde F(M,x)=F(D^2\psi,x)-F(D^2\psi-M,x)$$
is a concave fully nonlinear operator. Thus, we can not apply the theory of the obstacle problem for the convex fully nonlinear operator in \cite{Lee98} to $\tilde F(M,x)$. 

Precisely, it is uncertain that we can have Lemma \ref{near zero} for $v$, which is that the blowup of $v$ at $x \in \Gamma(v)=\Gamma^\psi(u)$ near $0$ is of the form $\frac{c}{2}(x_n^+)^2$, for a positive constant $c$. 
Hence, in contrast with linear theory \cite{LPS}, we only have the regularity of the free boundary $\Gamma(u)$, not $\Gamma^\psi (u)$, see Theorem \ref{thm-main-1} and Corollary \ref{cor-main}.
\end{remark}

\begin{lemma}\label{blo1}
Assume $F$ satisfies (F1)-(F3) and (F4)'. Let $u\in P_1(M)$ with an upper obstacle $\psi$ such that 
$$0\in \partial \Omega(\psi), \quad \lim_{x \to 0, x\in \Omega(\psi)} F(D^2\psi(x),x)>f(0), \quad f\ge c_0>0 \text{ in } B_1,$$
and 
$$ F(D^2\psi,x) \ge c_0>0 \text { in }\Omega(\psi).$$
Then $u_0\in P_\infty (M).$
\end{lemma}

\begin{proof}
Let $u_{r_i}$ and $\psi_{r_i}$ be sequences of the rescaling functions converging to blowups, $u_0$ and $\psi_0$, respectively. First, we claim that
$$F(D^2 \psi_0,0) =F(D^2\psi(0),0)\chi_{\Omega(\psi_0) } \quad \text { in } \re^n,$$
where $\Omega(\psi_0)=\re^n \setminus (\{\psi_0=0\}\cap \{\nabla \psi_0=0\})$ and $F(D^2\psi(0),0):=\lim_{x \to 0, x\in \Omega(\psi)} F(D^2\psi(x),x)$.
Let $x$ be a point in $\Omega(\psi_0)$. Then, by $C^{1,\alpha}_{loc}$ convergence of $\psi_{r_i}$ to $\psi_0$, we know that there exist $\delta>0$ and $i_0$ such that $B_{\delta}(x) \subset \Omega(\psi_{r_i})$, for all $i \ge i_0$. Then, by the definition of rescalings $\psi_{r_i}$, we have $r_i x \in \Omega(\psi)$. Furthermore, $\psi\in C^{2,\alpha}(\Omega(\psi))$ implies strong convergence of $\psi_{r_i}$ to $\psi_0$ in $C^{2,\beta}(B_{\delta}(x))$ for some $0<\beta<\alpha$. Thus,
\begin{equation*}
F(D^2\psi_0(x),0)=\lim_{i \to \infty} F(D^2 \psi_{r_i}(x), r_i x)=\lim_{i \to \infty} F(D^2 \psi(r_i x), r_i x)=F(D^2\psi(0),0).
\end{equation*}

Next, we prove that $u_0$ is a solution of 
$$F(D^2 u_0,0) =f(0)\chi_{\Omega(u_0) \cap\{u_0 <\psi_0\} } + F(D^2\psi(0),0) \chi_{\Omega(u_0) \cap \{u_0=\psi_0\}}, \quad u_0\le \psi_0 \quad \text { in } \re^n.$$
The rescaling $u_{r_i}$ is a solution of
$$F(D^2 u_{r_i}, r_ix) =f(r_ix)\chi_{\Omega(u_{r_i}) \cap\{u_{r_i}<\psi_{r_i}\} } + F(D^2\psi_{r_i},r_i x) \chi_{\Omega(u_{r_i}) \cap \{u_{r_i}=\psi_{r_i}\}}, \quad u_{r_i}\le \psi_{r_i} \quad \text { in } B_{1/r_i},$$
where $\Omega(u_{r_i}):=B_{1/r_i} \setminus (\{u_{r_i}= 0\} \cap \{\nabla u_{r_i}=0\})$. Let $x$ be a point in $\Omega(u_0)\cap\{u_{0}<\psi_{0}\}$. Then, by $C^{1,\alpha}_{loc}$ convergence of $u_{r_i}$ to $u_0$, there exist $\delta>0$ and $i_0$ such that $B_{\delta}(x) \subset \Omega(u_{r_i})\cap\{u_{r_i}<\psi_{r_i}\}$, for all $i \ge i_0$. Then 
$$F(D^2u_{r_i}(y))=f(r_ix) \quad \text{ in } B_\delta(x).$$
Since $f\in C^{0,\alpha} (B_1)$, we have $C^{2,\alpha}$ estimates for $u_{r_i}$ and we may assume strong convergence of $u_{r_i}$ to $u_0$ in $C^{2,\beta}(B_{\delta}(x))$ for some $0<\beta<\alpha$. Thus we have $|D^2 u_0(x)|\le M$ and 
\begin{equation*}
F(D^2u_0(x),0)=\lim_{i \to \infty} F(D^2u_{r_i}(x), r_ix)=\lim_{i \to \infty} f(r_i x)=f(0)\ge c_0>0. %\quad \forall x\in G(v_0)
\end{equation*}
Since $F(0,0)=0$, we know that $ F(D^2u_0,0)=0$ a.e. on $\{u_0=0\}$. Moreover, $\Omega(u_0) \cap \{u_0=\psi_0\}\subset \Omega(\psi_0)$ implies that $F(D^2u_0,0)=F(D^2 \psi_0,0)$ in $\Omega(u_0) \cap \{u_0=\psi_0\}$.

Therefore, $u_0$ is a solution of 
$$F(D^2 u_0,0) =f(0)\chi_{\Omega(u_0) \cap\{u_0 <\psi_0\} } + F(D^2\psi(0),0) \chi_{\Omega(u_0) \cap \{u_0=\psi_0\}}, \quad u_0\le \psi_0 \quad \text { in } \re^n,$$
with $f(0)< F(D^2\psi(0),0)$. %\rred{The optimal regularity implies that .} 
Furthermore, by the non-degeneracy, Lemma \ref{nond}, we have $0\in \Gamma(u_0)$, see Remark \ref{rere}. Consequently, $u_0$ is in $P_\infty(M)$ for the fully nonlinear operator $G(M)=F(M,0)/f(0)$.
%For $x\in \{v_0>0\}$, by the strong convergence of $v_{r_i}$ to $v_0$ in $C^2(B_\delta(x))$ for some $\delta>0$, we know . 
\end{proof}

\subsection{Classification of Global Solutions}\label{sec-cla}

In this subsection, we discuss the classification of global solutions, which means that the global solution $u\in P_\infty(M)$ with the upper obstacle $\psi=\frac{a}{2}(x^+_n)^2$ ($a>1$) and the thickness assumption \eqref{thi psi zero}, is $\frac{a}{2}(x^+_n)^2$ or $\frac{1}{2}(x^+_n)^2$.

First, we observe that the zero set of $u$ contains a half plain $\{x_n<0\}$, by the thickness assumption \eqref{thi psi zero} and the non-degeneracy, Lemma \ref{nond}. Thus, the optimal ($C^{1,1}$) regularity for the solution $u\in P_\infty(M)$ implies that $\partial_e u/x_n$ is finite. Moreover, by considering the rescaling functions of $u$ and $\psi$ with a distance from $x\in \{x_n>0\}$ to $\{x_n=0\}$, \eqref{eq resc}, we show that $\partial_e u\equiv 0$, for any direction $e \in \mathbb{S}^{n-1}\cap e_n^{\perp}$. It implies that $u$ is one-dimensional, and it is $\frac{a}{2}(x^+_n)^2$ or $\frac{1}{2}(x^+_n)^2$.

%of the form, $u=c(x_n^+)^2$, for a positive constant $c$.

\begin{prop}\label{cla} 
Assume $F=F(M)$ satisfies (F1)-(F2). Let $u\in P_\infty(M)$ be a solution of 
$$F(D^2 u) = \chi_{\Omega(u)\cap \{u<\psi\}} + a \chi_{\Omega(u) \cap \{u=\psi\}}, \quad u \le \psi \quad \text{ a.e. in } \re^n,$$ 
with the upper obstacle
$$\psi(x)=\frac{a}{2}( x_n^+)^2,$$
for a constant $a>1$. Suppose 
$$\delta_r(u, \psi) \ge 0, \quad \text{ for all } r >0.$$
Then, we have
$$u(x)=\frac{1}{2}( x_n^+)^2 \quad or \quad u (x)=\frac{a}{2}( x_n^+)^2.$$
\end{prop}

\begin{proof}
The condition $u \le \psi=\frac{a}{2}( x_n^+)^2$ on $\re^n$ implies that $u(x)\le 0$ on $\{x_n \le 0\}.$ We claim that $\{x_n<0\} \subset \Lambda(u)$. 
First, we suppose that $\partial \Omega(u)\cap \{x_n<0\} \neq \emptyset$. Then, by non-degeneracy, (Lemma \ref{nond}), we have that $\{u>0\}\cap \{x_n<0\} \neq \emptyset$ and we arrive at a contradiction. Next, we suppose that $\{x_n<0\} \subset \Omega(u).$ Since $\{\psi=0\}=\Lambda(\psi)=\{x_n\le 0\}$, it is a contraction to $\delta_r(u, \psi) \ge 0, \text{ for all } r >0.$ 
% Then we know that $u<0$ on $\{x_n<0\}$. 

Therefore, we have that $\{x_n<0\}$ is contained in $\Lambda(u)$. %In other words, we have
%\begin{equation*}
%\Omega(u) \subset \{x_n>0\}
%\end{equation*}
%and 
Hence, $u=0$ on $\{x_n\le 0\}$ and $\partial_e u=0$ on $\{x_n\le 0\}$ for all $e\in \mathbb{S}^{n-1}\cap e_n^{\perp}$, where $\mathbb{S}^{n-1}:=\{x\in \re^n : |x|=1\}$ and $e^\perp :=\{x\in \re^{n} : x \perp e\}$ for $e\in \mathbb{S}^{n-1}$.

In order to have the conclusion, it suffices to show that $\partial_e u \equiv 0$ on $\re^n$, for any direction $e\in \mathbb{S}^{n-1}\cap e_n^{\perp}$. Thus, we fix $e_1\in \mathbb{S}^{n-1}\cap e_n^{\perp}$ and define 
$$0\le \sup_{x\in \{x_n>0\}} \frac{\partial_1 u(x)}{x_n}=:M_0.$$

By the optimal regularity and $\partial_1 u=0$ on $\{x_n\le 0\}$, we know that $M_0$ is finite. If we prove that $M_0$ is $0$, we have that $\partial_1u \equiv 0$ on $\re^n$. Since the direction $e_1$ is arbitrary, we have that $\partial_e u=0$ on $\{x_n\le 0\}$, for all $e\in \mathbb{S}^{n-1}\cap e_n^{\perp}$.

Arguing by contradiction, suppose $M_0>0$. Since $\partial_1 u\equiv 0$ on $(\Omega(u) \cap \{u< \psi\})^c$, we can take a sequence $x^j \in \Omega(u) \cap \{u< \psi\} \subset \{x_n>0\}$ such that 
%$$\lim_{j \to \infty}\frac{1}{x^j_n} \partial_{e^{j}} u(x^j)=M_0.$$ 
%We may assume that $e^j \to e_1$ up to subsequence, Then
%Moreover, up to a subsequence, there exists $e^0 \in \mathbb{S}^{n-1}\cap e_n^{\perp}$ such that $e^j \to e^0 \in \mathbb{S}^{n-1}\cap e_n^{\perp}$ and
%\begin{align*}
%\left|\frac{1}{x^j_n} \nabla' u(x^j) \cdot e^0-M_0 \right|&\le \left|\frac{1}{x^j_n} \nabla' u(x^j) \cdot (e^0-e^j) \right|+\left|\frac{1}{x^j_n} \nabla' u(x^j) \cdot e^j-M_0 \right|\\
%&\le C|e-e^j|+\left|\frac{1}{x^j_n} \nabla' u(x^j) \cdot e^j-M_0 \right| \to 0,
%\end{align*}
%as $j \to \infty$. Thus, up to a rotation, we may assume that $e^0=e_1$ and we obtain
$$\lim_{j \to \infty}\frac{1}{x^j_n} \partial_{1} u(x^j)=M_0.$$

Let $r_j:=x_n^j=|x_n^j|$ and consider rescaling functions
\begin{equation}\label{eq resc}
u_{r_j}(x):=\frac{u\left(((x^j)',0)+r_jx\right)}{(r_j)^2} \quad \text{ and } \quad \psi_{r_j}(x):=\frac{\psi\left(((x^j)',0)+r_jx\right)}{(r_j)^2}=\psi(x).
\end{equation}
Then, $D^2 u_{r_j}$ are uniformly bounded and $u_{r_j}\equiv 0$ on $\{x_n\le 0\}$. Thus,
$$u_{r_j}(x) \to \tilde u(x) \text{ in } C^{1,\alpha}_{loc}(\re^n) \text{ for any } \alpha \in [0,1),$$
\begin{equation}\label{ousub}
\tilde u \equiv 0 \quad \text{ on } \{x_n\le 0\} %(\Omega (\tilde u) \subset \{ x_n>0\})
\end{equation}
and $\tilde u$ is a solution of
$$F(D^2 u) = \chi_{\Omega(u) \cap \{u<\psi\}} + a \chi_{\Omega(u) \cap \{u=\psi\}}, \quad u \le \psi \quad \text{ a.e. in } \re^n,$$ 
with the upper obstacle
$$\psi(x)=\frac{a}{2}( x_n^+)^2.$$
% Since $\partial_1 \psi\equiv 0 $ on $\re^n$ and 
By the definition of $M_0$, for $x\in \{x_n>0\}$,
$$\partial_1 u_{r_j}(x)=\frac{\partial_{1} u\left(((x^j)',0)+r_jx\right)}{r_jx_n}\cdot x_n\le M_0 x_n.$$
Hence, we have $\partial_{1} \tilde u(x) \le M_0 x_n$ on $\{x_n>0\}$.
Moreover,
$$\partial_1 \tilde u(e_n)=\lim_{j\to \infty} \partial_1 u_{r_j}(e_n)=\lim_{j\to \infty}\frac{\partial_1 u\left(((x^j)',0)+r_je_n\right)}{r_j}=\lim_{j\to \infty} \frac{\partial_1 u(x^j)}{x_n^j}=M_0.$$

If $e_n \in (\Omega(\tilde u) \cap \{\tilde u< \psi\})^c$, then $\partial_1 \tilde u (e_n)=0$ and we arrive at a contradiction. Thus, $e_n \in \Omega(\tilde u) \cap \{\tilde u< \psi\}$. Let $\tilde \Omega(\tilde u)$ be the connected component of $\Omega(\tilde u) \cap \{\tilde u< \psi\}$ containing $e_n$. By \eqref{ousub}, we know that $\tilde \Omega(\tilde u)\subset \Omega(\tilde u) \subset \{x_n>0\}$.

By differentiating $F(D^2\tilde u)=1$ on $\tilde \Omega(\tilde u)$ with respect to $e_1$, we have $F_{ij}(D^2\tilde u) \partial_{ij} \partial_{1} \tilde u=0$ and $F_{ij}(D^2\tilde u) \partial_{ij} \partial_{1} (\tilde u- M_0x_n)=0$ on $\tilde \Omega(\tilde u)$. Thus, the strong maximum principle implies that
$$\partial_{1}\tilde u=M_0x_n\quad \text{ on } \tilde \Omega(\tilde u)\subset \{x_n>0\}.$$

If there exists $x\in \partial \tilde \Omega(\tilde u) \cap \{x_n>0\}$, then $\partial_{1} \tilde u(x)=0=Mx_n$ and we have a contradiction, i.e., we have $\partial \tilde \Omega(\tilde u) \cap \{x_n>0\} = \emptyset$. Then, $\tilde \Omega(\tilde u)\subset \{x_n>0\}$ implies $\tilde \Omega(\tilde u)=\{x_n>0\}$, $\partial_{1} \tilde u \equiv M x_n$ on $\{x_n>0\}$ and
$$\tilde u(x)=M_0x_1x_n+g(x_2,...,x_n)\quad \text{ on } \{x_n>0\}.$$
Since $\tilde u$ is in $C^{1,1}_{loc}(\re^n)$ and $\tilde u \equiv 0 \text{ on } \{x_n\le 0\}$, we have
$$\partial_n \tilde u(x)=M_0x_1+\partial_n g(x_2,...,x_n)=0\quad \text{ on } \{x_n=0\}$$
and it does not hold unless $M_0=0$ and $\partial_n g(x_2,...,0)=0$. Hence, we arrive at a contradiction.
\end{proof}

\subsection{Directional Monotonicity and proof of Theorem \ref{thm-main-2}}\label{sec-dir}
In this subsection, we show the directional monotonicity for solutions of \eqref{eq-main-local} and the regularity of the solutions $u \in P_1(M)$. We note that the argument for the linear case is discussed in \cite{LPS}. %and for the single obstacle problem for the fully nonlinear operator, it is explained in \cite{Lee98, FS14, IM16a}. 

%Thus we omit or summarize some proofs in this subsection. 
\begin{lemma}\label{inc dri}
Assume $F\in C^1$ satisfies (F1)-(F3) and (F4)' and let $u$ be a solution of
\begin{equation*}
F(D^2u,rx) =f(rx)\chi_{\Omega(u) \cap \{u<\psi\}} + F(D^2\psi(rx),rx) \chi_{\Omega(u)\cap \{0<u=\psi\}}, \quad u\le \psi \quad \text { in } B_1
\end{equation*}
and assume that $f(x) \ge c_0 >0$ in $B_1$, Suppose that we have
$$C\partial_e \psi -\psi \ge 0, \quad C\partial_e u -u\ge -\epsilon_0 \quad \text{ in } B_1,$$
for a direction $e$ and $\epsilon_0<c/64\lambda_1 n$. Then we obtain
$$C\partial_e u -u \ge 0 \quad \text{ in } B_{1/4},$$
if $0<r\le r'_0$, for some 
$$r'_0=r'_0 (C ,c_0,\norm{\nabla F}_{L^\infty(B_M\times B_1)},\norm{\nabla f}_{L^\infty(B_1)} ).$$
\end{lemma}

\begin{proof}
By differentiating $F(D^2u,rx) =f(rx)$ with respect to the direction $e$, we have
\begin{equation}\label{der u}
F_{ij}(D^2u(x),x) \partial_{ij} \partial_e u(x) = r \partial_e f(rx)-r(\partial_{e} F)(D^2u(x) ,rx) \quad \text{ in } \Omega (u)\cap \{u<\psi\},
\end{equation}
where $\partial_{e} F$ is the spatial directional derivative of $F(M,x)$ in the direction $e$.

Arguing by contradiction, suppose there is a point $y\in B_{1/2} \cap \Omega(u)\cap \{u<\psi\}$ such that $C\partial_e u(y) -u(y) < 0$. We consider the auxiliary function
$$\phi(x)=C\partial_e u(x) -u(x)+\frac{c_0}{4\lambda_1n}|x-y|^2.$$

By Proposition 2.13 of \cite{CC} and the condition $(F1)$ ($F(0,x)=0$ for all $x \in \re^n$), we have $u \in S(\lambda_0/n, \lambda_1, f(rx))$ in $\Omega (u)\cap \{u<\psi\}$ and moreover \eqref{der u} implies $\partial_e u \in S(\lambda_0/n, \lambda_1, r \partial_e f(rx)-r(\partial_{e} F)(D^2u(x) ,rx) )$. Hence, we have 
$$\phi \in \overline{S}(\lambda_0/n, \lambda_1, r \partial_e f(rx)-r(\partial_{e} F)(D^2u(x) ,rx) -f(rx)+c_0/2) \text{ in } \Omega (u)\cap \{u<\psi\}.$$
Furthermore, since there is a sufficiently small constant $\tilde r_0=\tilde r_0(C, c_0, \norm{\nabla f}_{L^\infty(B_1)},$ $\norm{\nabla F}_{L^\infty(B_M\times B_1)})>0$ such that
\begin{align*}
Cr \partial_e f(rx)&-Cr(\partial_{e} F)(D^2u(x) ,rx)-f(rx)+c_0/2 \\
&\le Cr\norm{\nabla f}_{L^\infty(B_1)}+ Cr\norm{\nabla_x F}_{L^\infty(B_M\times B_1)}-f(rx)+c_0/2 \le 0,
\end{align*}
for all $r<\tilde r_0,$ we have
$$\phi(x)\in \overline S (\lambda_0/n, \lambda_1,0) \quad \text{ in }\Omega (u)\cap \{u<\psi\}, \quad \text{ for all } r \le \tilde r_0.$$
By the minimum principle of $\phi$ in $B_{1/4}(y) \cap \Omega(u)\cap \{u<\psi\}$, we have
$$\inf_{\partial (B_{1/4}(y) \cap \Omega(u)\cap \{u< \psi\})} \phi \le \phi(y)<0.$$
Moreover, $C\partial_e \psi -\psi \ge 0\text{ in } B_1$ implies that $\phi \ge0$ on $\partial \left( \Omega(u)\cap \{u< \psi\} \right)$. Thus, we obtain
$$\inf_{\partial B_{1/4}(y) \cap ( \Omega(u)\cap \{u< \psi\})} \phi <0 \quad \text{ and } \quad \inf_{\partial B_{1/4}(y) \cap( \Omega(u)\cap \{u< \psi\})}\left( C\partial_e u-u \right) <-\frac{c_0}{128\lambda_1n}.$$
Since $\epsilon_0<\frac{c_0}{64\lambda_1n}$, we have a contradiction. 
\end{proof}

\begin{lemma}[Directional monotonicity]\label{dir mon}
Let $u$, $\psi$, $F$ be as in Theorem \ref{thm-main-2}. Then for any $\delta \in (0,1]$, there exists 
$$r_\delta=r_\delta(u,\psi, c_0, \norm{\nabla F}_{L^\infty(B_M\times B_1)}, \norm{F}_{L^\infty(B_M\times B_1)},\norm{\nabla f}_{L^\infty(B_1)})>0$$ such that
\begin{align*}
u &\ge 0 \quad \text{ in } B_{r_1}\\
\partial_e u &\ge 0 \quad \text{ in } B_{r_\delta}\quad \text{ for any } e\in C_\delta \cap \partial B_1,
\end{align*}
where
$$C_\delta=\{x\in \re^n : x_n > \delta |x'|\}, \quad x'=(x_1,...,x_{n-1}).$$
\end{lemma}

\begin{proof}
We denote $u_r$ and $\psi_r$ by rescalings of $u_{r_i}$ and $\psi_{r_i}$, respectively. The thickness assumption \eqref{thi psi0 zero0} implies that $\psi_0=\frac{a}{2}(x_n^+)^2$, $(a>1)$, in an appropriate system of coordinates, see e.g. Proposition 4.7 of \cite{LP}. Then, by Proposition \ref{cla}, we know that $u_0$ is $\frac{1}{2}(x_n^+)^2$ or $\frac{a}{2}(x_n^+)^2$. Hence, for any $e\in C_\delta=\{x\in \re^n : x_n > \delta |x'|\}, x'=(x_1,...,x_{n-1})$, 
we obtain 
$$ \delta^{-1} \partial_e \psi_0-\psi_0 \ge 0 \quad \text{ and }\quad \delta^{-1} \partial_e u_0-u_0 \ge 0 \quad \text{ in } \re^n .$$
By the $C^{1, \alpha}$ convergence of $\psi_r$ and $u_r$ to $\psi_0$ and $u_0$, respectively, %$\norm{\psi_r-\psi_0}_{C^{1,\alpha}(B_1)} \to 0$, $\norm{u_r-u_0}_{C^{1,\alpha}(B_1)} \to 0$, 
we obtain that 
$$\delta^{-1} \partial_e \psi_r-\psi_r \ge -\epsilon_0 \quad \text{ and } \quad \delta^{-1} \partial_e u_r-u_r \ge -\epsilon_0 \quad \text{ in } B_1,$$
for $\epsilon_0<c/64\lambda_1 n$ and $r<\hat r_\delta (u, \psi)$. Then, by applying the directional monotonicity for the solution of the single obstacle, Lemma 13 of \cite{IM16a} to $\psi$, we have that $\delta^{-1} \partial_e \psi_r-\psi_r \ge 0$ in $B_{1/2}$ for all $r<r'_\delta(u,\psi, c_0,\norm{\nabla F}_{L^\infty(B_M\times B_1)},\norm{\nabla f}_{L^\infty(B_1)})$. Furthermore, by Lemma \ref{inc dri}, we have 

\begin{equation}\label{direc ur}
\delta^{-1} \partial_e u_r -u_r \ge 0 \quad \text{ in } B_{1/4},
\end{equation}
for $0<r\le \tilde r_\delta=\tilde r_\delta (u, \psi, c_0,\norm{\nabla F}_{L^\infty(B_M\times B_1)},\norm{\nabla f}_{L^\infty(B_1)} )$.

We claim that $u_r=0$ in $\{x_n<-1/8\}\cap B_{1/4}$. By the $C^{1,\alpha}$ convergence of $u_r$ to $u$, we may assume that $\norm{u_r-u_0}_{L^\infty(B_1)}\le \frac{c}{8\lambda_1n}\times \frac{1}{128}$ for $0<r\le \tilde r_\delta$. Let $x_0$ be a point in $\{x_n< 0\}\cap \Omega(u_r) \cap B_{1/4}$. By the non-degeneracy, Lemma \ref{nond}, we have 
\begin{equation}\label{nond ur}
\sup_{\partial B_\rho(x_0)}u_r \ge u_r(x_0)+\frac{c}{8\lambda_1n}\rho^2,
\end{equation}
where $\rho: = |(x_0)_n|$. Since $u_0=0$ in $\{x_n\le 0\}$ implies $\norm{u_r-u_0}_{L^\infty(\partial B_\rho(x_0))}=\norm{u_r}_{L^\infty(\partial B_\rho(x_0))}\le \frac{c}{8\lambda_1n}\times \frac{1}{128}$, by \eqref{nond ur}, we have $\frac{c}{8\lambda_1n}\rho^2 \le \frac{c}{4\lambda_1n}\times \frac{1}{128} $ and $\rho \le \frac{1}{8}$. Therefore, $\{x_n< 0\}\cap \Omega(u_r)\cap B_{1/4} \subset \{-1/8 < x_n<0\}$ and $u_r=0$ in $\{x_n<-1/8\}\cap B_{1/4}$.

Let us fix $\delta=1$ and let $v(x)=exp(- e \cdot x) u_r(x)$. Then, by \eqref{direc ur}, we have $\partial_e v\ge 0$ in $B_{1/4}$ for any $e\in C_\delta$. Thus, $u_r=0$ in $\{x_n<-1/8\}\cap B_{1/4}$ implies that $v=0$ in $\{x_n<-1/8\}\cap B_{1/4}$, $v\ge 0$ in $B_{1/4}$, and $u_r \ge 0$ in $B_{1/4}$ for all $r\le \tilde r_1$. By the definition of scaling functions $u_r$, we have $u \ge 0$ in $B_{r_1}$, for $r_1=\frac{1}{4\tilde r_1}$. Furthermore, \eqref{direc ur} implies that 
$$\partial_e u \ge 0 \quad \text{ in } B_{r_\delta}\quad \text{ for any } e\in C_\delta \cap \partial B_1,$$
for $r_\delta=\frac{1}{4\tilde r_\delta}<r_1, \delta \in (0,1]$.
\end{proof}

%\subsection{Proof of Theorem \ref{thm-main-2}}\label{sec reg loc}

\begin{lemma}\label{pos u}
Let $u$, $\psi$, $F$ be as in Theorem \ref{thm-main-2}. Then there exists 
$$r_1=r_1(u,\psi, c_0, \norm{\nabla F}_{L^\infty(B_M\times B_1)}, \norm{F}_{L^\infty(B_M\times B_1)},\norm{\nabla f}_{L^\infty(B_1)})>0$$ such that $u$ is a solution of
$$F(D^2 u,x) =f\chi_{\{ 0< u < \psi \}} +F(D^2 \psi,x) \chi_{\{0< u = \psi\}}, \quad 0\le u \le \psi \quad \text{ in } B_{r_1}.$$
%and $v:=\psi-u$ is a solution of 
%$$\La v =(\La \psi-f)\chi_{\{ 0< v < \psi \}} +\La \psi \chi_{\{0< v = \psi\}} \quad 0\le v \le \psi \quad \text{ in } B_{r_1}.$$
Moreover, if $u_0$ and $\psi_0$ are blowup functions of $u$ and $\psi$ at $0$, respectively, then in an appropriate system of coordinates,
$$ \psi_0(x)=\frac{a}{2}( x_n^+)^2\quad \text{ and } \quad u_0(x)=\frac{1}{2} (x_n^+)^2.$$
\end{lemma}

\begin{proof}
By Lemma \ref{dir mon}, there is $r_1=r_1(u, \psi, c_0, \norm{\nabla F}_{L^\infty(B_M\times B_1)}, \norm{F}_{L^\infty(B_M\times B_1)},\norm{\nabla f}_{L^\infty(B_1)})>0$ such that $u\ge0$ in $B_{r_1}$. Hence $u$ is a solution of
$$F(D^2 u,x) =f\chi_{\{ 0< u < \psi \}} +F(D^2 \psi,x) \chi_{\{0< u = \psi\}}, \quad 0\le u \le \psi \quad \text{ in } B_{r_1}.$$
and $v:=\psi-u$ is a solution of 
$$\tilde F(D^2v,x)=\left(F(D^2 \psi,x)-f\right)\chi_{\{0< v <\psi\}}+F(D^2 \psi,x) \chi_{\{0<v=\psi\}} \quad 0\le v \le \psi \quad \text{ in } B_1,$$
where $\tilde F (M,x):= F(D^2\psi,x)-F(D^2 \psi-M, x)$. Since $0\le v \le \psi$, we have that $\{v>0\} \subset \{\psi>0\}=\Omega(\psi)$. Thus, $\min\{F(D^2 \psi,x) , F(D^2 \psi,x)-f \}\ge c_0>0$ in $\Omega(\psi)$ implies
$$\tilde F(D^2 v,x)=(F(D^2 \psi,x)-f)\chi_{\{ 0< v < \psi \}} +F(D^2 \psi,x) \chi_{\{0< v = \psi\}}\ge c_0>0 \quad \text{ in } \{v>0\}.$$
Thus, by the same argument in Lemma \ref{nond}, we have the non-degeneracy for $v$, 
\begin{equation*}
\sup_{\partial B_r(x)}v \ge v(x)+\frac{\lambda}{8n}r^2 \quad x\in \overline{\Omega(v)}\cap B_{r_0},
\end{equation*}
for $B_r(x) \Subset B_{r_0}$.
This implies $0 \in \Gamma(v_0)=\Gamma^{\psi_0}(u_0)$, where $v_0$ is a blowup functions of $v$ at $0$ such that $v_0=\psi_0-u_0$, see Remark 3.2. Consequently, we have
$$ \psi_0(x)=\frac{a}{2}( x_n^+)^2 \quad \text{ and } \quad u_0(x)=\frac{1}{2} (x_n^+)^2 $$
in an appropriate system of coordinates. 
%Thus by the same way for $u$, we have the directional monotonicity for $v$, i.e.,
%Then for any $\delta \in (0,1]$ there exists $r'_\delta=r'_\delta(v)>0$ such that
%\begin{align*}
% v &\ge 0 \quad \text{ in } B_{r'_\delta}\\
%\partial_e v &\ge 0 \quad \text{ in } B_{r'_\delta}\quad \text{ for any } e\in C_\delta.
%\end{align*}
\end{proof}

By the uniqueness of the blowup for the single obstacle problem, we have the uniqueness of blowup for $\psi$, i.e., for any sequence $\lambda \to 0$, 
$$\psi_{\lambda} \to \psi_0=\frac{1}{2}(x_n^+)^2 \quad \text{ in } C^{1,\alpha}_{loc}(\re^n),$$ 
in an appropriate system of coordinates. Then, the uniqueness of the blowup for $u$ directly follows from Lemma \ref{pos u}.
%By Lemma \ref{dir mon} and Lemma \ref{pos u}, we have the uniqueness of blowup. 
\begin{prop}[Uniqueness of blowup]
Let $u, \psi, F$ be as in Theorem \ref{thm-main-2}. Then the blowup function of $u$ at $0$ is unique, i.e., in an appropriate system of coordinates,
for any sequence $\lambda \to 0$, 
$$u_{\lambda} \to u_0=\frac{1}{2}(x_n^+)^2 \quad \text{ in } C^{1,\alpha}_{loc}(\re^n)$$ 
as $\lambda \to 0$.
\end{prop}

In the following lemma, we have that the blowups of $u$ for any points $x$ near $0$ are also half-space functions.
\begin{lemma}\label{near zero}
Let $u, \psi, F$ and $r_1$ be as in Theorem \ref{thm-main-2}. Then there is $r'_1=r'_1(u,\psi)>0$ such that the blowup function of $u$ at $x\in \Gamma(u) \cap B_{r'_1}$ are half-space functions.
\end{lemma}

\begin{proof}
%By Lemma \ref{dir mon} and \ref{pos u}, we have 
By the directional monotonicity for $u$ and $\psi$ (Lemma \ref{dir mon} and \ref{pos u}), we have that,
%and the fact that $u$ is nonnegative in $B_{r_1}$ for some $r_1=r_1(u, c_0, \norm{\nabla F}_{L^\infty(B_M\times B_1)}, \norm{F}_{L^\infty(B_M\times B_1)},$ $\norm{\nabla f}_{L^\infty(B_1)})>0$. Thus $u$ is a solution of 
%$$F(D^2 u,x) =f\chi_{\{ 0< u < \psi \}} +F(D^2 \psi,x) \chi_{\{0< u = \psi\}} \quad 0\le u \le \psi \quad \text{ in } B_{r_1}$$
for any $\delta \in (0,1]$, there exists 
$$r_\delta=r_\delta(u, c_0, \norm{\nabla F}_{L^\infty(B_M\times B_1)}, \norm{F}_{L^\infty(B_M\times B_1)},\norm{\nabla f}_{L^\infty(B_1)})>0$$ such that
$r_1\ge r'_\delta=r'_\delta(u,\psi)>0$ and
\begin{align*}
\psi,u &\ge 0 \quad \text{ in } B_{r'_1}\\
\partial_e \psi, \partial_e u &\ge 0 \quad \text{ in } B_{r'_\delta}\quad \text{ for any } e\in C_\delta.
\end{align*}
%With out loss of generality, we assume $r'_1\ge r_1$.\\
Then, the free boundaries $\partial \{u=0\}\cap B_{r'_1}=\Gamma(u) \cap B_{r'_1}, \partial \{\psi=0\}\cap B_{r'_1}$ are represented by Lipschitz functions. %see e.g. Proposition 4.8 of \cite{PSU}.

Let $x^0$ be a point in $\Gamma(u) \cap B_{r'_1}$ and assume that there exists $r_0>0$ such that 
$$\{u=\psi\}\cap B_r(x^0)\neq \emptyset \quad \text{ for all } r< r_0.$$
Then, there is a sequence of points $x^j$ such that $x^j\in \{u=\psi\}$ and $x^j \to x^0$ as $j \to \infty.$ Thus,
$$\psi(x^j)=u(x^j)\to 0 \quad \text{ as } j \to \infty,\quad \text{ and } \quad x^0\in \{\psi=0\}.$$
Since $u$ is nonnegative in $B_{r'_1}$, we have $0\le u \le \psi$ and $\{\psi=0\}\subset \{u=0\}$ in $B_{r'_1}$. Thus, $x^0\in\Gamma(u) \cap B_{r'_1}$ implies $x^0\in \partial \{\psi=0\}$. Furthermore, by the Lipschitz regularity of the zero set of $\psi$, $\{\psi=0\}$ and the positivity of $u$, ($0\le u\le \psi$), we obtain
\begin{equation*}
\delta_r(u, \psi,x_0) = \delta_r(\psi,x_0) \ge \epsilon_0 \quad \text{ for all } r< 1/4.
\end{equation*}
Then, by classification of blowup, Proposition \ref{cla}, we know that the blowup of $u$ at $x^0$ is a \emph{half-space solution}, which means that it is of the form $\frac{c}{2}(x_n^+)^2$, for a positive constant $c$.

Next, we assume that, for $x^0\in \Gamma(u)\cap B_{r'_1}$, there exists $r_0>0$ such that 
$$\{u=\psi\}\cap B_{r_0}(x^0)= \emptyset.$$
Then $u$ is a solution of
$$F(D^2 u,x) = f\chi_{\{u>0\}}, \quad u\ge0 \quad \text{ in } B_{r_0}(x^0).$$
On the other hand, Lipschitz regularities of $\Gamma(u)$ implies the thickness assumption for $u$ near $x_0$. Then, %by Theorem 3.9 of \cite{IM14}, 
the blowup function of $u$ at $0$ is a half-space solution.
\end{proof}

By using lemmas in this subsection and arguments in \cite{Lee98, PSU, IM16a, LPS}, we prove one of the main theorems of the paper, Theorem \ref{thm-main-2}.

%to prove the main theorem and the lest of proof follows arguments 

%Hence we present the sketch of the proof of .

\begin{proof}[Proof of Theorem \ref{thm-main-2}]
The directional monotonicity for $u$, Lemma \ref{dir mon}, implies that the free boundary $\Gamma(u)\cap B_{r\delta/2}$ is represented as a graph $x_n=f(x')$ for Lipschitz function $f$ and the Lipschitz constant of $f$ is less than $\delta$ in $B_{r\delta/2}$. Since $\delta>0$ can be chosen arbitrary small, we have a tangent plane of $\Gamma(u)$ and the normal vector $e_n$ at $0$. By Lemma \ref{near zero}, for any $z \in \Gamma(u) \cap B_{r'_1}$, we know that the blowup is the half-space solution and there is a tangent plane for $z\in \Gamma(u) \cap B_{r'_1}$ with tangent vector $\nu_z$. By Lemma \ref{dir mon}, for $z\in \Gamma(u) \cap B_{r_\delta}$, we have $\nu_z \cdot e \ge 0$ for any $e\in C_\delta$. Hence, $\nu_z$ is in $C_{1/\delta}$ and then, for sufficiently small $\delta>0$, $\nu_z$ is close to $e_n$. Specifically, we have that
$$|\nu_z-e_n| \le C \delta, \quad z\in \Gamma(u) \cap B_{r_\delta}.$$
Therefore, $\Gamma(u)$ is $C^1$ at $0$ and by the same argument, $\Gamma(u) \cap B_{r'_1}$ is $C^1$.
\end{proof}

%\section*{Acknowledgement}
%\acks
%We thank Henrik Shahgholian for helpful discussions on the double obstacle problem. Ki-Ahm Lee also holds a joint appointment with the Research Institute of Mathematics of Seoul National University. This work was supported by the National Research Foundation of Korea (NRF) grant funded by the Korean government (MSIP) [NRF-2020R1A2C1A01006256 to K.-A.L.]; and Basic Science Research Program through the National Research Foundation of Korea (NRF) funded by the Ministry of Education [2020R1A6A3A01099425 to J.P.].

%the National Institutes of Health [AA123456 to C.S., BB765432 to M.H.]; and the Alcohol & Education Research Council [hfygr667789].'

%Ki-Ahm Lee has been supported by the National Research Foundation of Korea (NRF) grant funded by the Korean government (MSIP) (No. NRF-2020R1A2C1A01006256). Ki-Ahm Lee also holds a joint appointment with the Research Institute of Mathematics of Seoul National University. This research was supported by Basic Science Research Program through the National Research Foundation of Korea(NRF) funded by the Ministry of Education(2020R1A6A3A01099425).

%
\section*{Acknowledgement}
We thank Henrik Shahgholian for helpful discussions on the double obstacle problem. 
Ki-Ahm Lee has been supported by the National Research Foundation of Korea (NRF) grant funded by the Korean government (MSIP) (No. NRF-2020R1A2C1A01006256). Ki-Ahm Lee also holds a joint appointment with the Research Institute of Mathematics of Seoul National University. This research was supported by Basic Science Research Program through the National Research Foundation of Korea(NRF) funded by the Ministry of Education(2020R1A6A3A01099425).

%-----------------------------The-end-of-document--------------------------------

%\bibliography{References}
%\bibliographystyle{chicago}

%-----------------------------The-end-of-reference---------------------------------

\end{document}